\documentclass{article}
\usepackage{maa-monthly}

\theoremstyle{theorem}
\newtheorem{theorem}{Theorem}
\newtheorem*{lemma}{Lemma}

\theoremstyle{definition}

\begin{document}

\title{Extension of the Bertrand--De Morgan Test and Its Application}
\markright{Notes}
\markright{Extension of the Bertrand-De Morgan test}
\author{Vyacheslav M. Abramov}

\maketitle

\begin{abstract}
We provide a simple proof for the extended Bertrand--De Morgan test that was earlier studied in [\v{D}uri\v{s}, F. Infinite Series: Convergence tests. Bachelor thesis, 2009] and [Tabatabai Adnani, A. A. Reza, A., Morovati, M. (2013), \textit{J. Lin. Topol. Algebra}, 2(3): 141--147]
and demonstrate an application of that test in the theory of birth-and-death processes.
\end{abstract}

{\centering Dedicated to the memory of Robert Liptser (1936--2019).

}
%

\section{Introduction.}

Let
\begin{equation}\label{0}
\sum_{n=1}^\infty a_n
\end{equation}
be a series of positive numbers. The simplest test for convergence or divergence of series \eqref{0} is the ratio test. The first ratio test, the most elementary, was due to d'Alembert \cite{d'Al}. According to the d'Alembert test (also known as the Cauchy test), the series \eqref{0} converges if $a_{n+1}/a_n<l<1$ for all large $n$ and diverges if $a_{n+1}/a_n\geq1$ for all large $n$. The more general tests obtained by Raabe, Gauss, Bertrand, De Morgan, and Kummer, which required weaker conditions for convergence, were classified into the De Morgan hierarchy of ratio tests \cite{Bl, Br}.
In the present note we discuss the Bertrand--De Morgan test \cite{W} and its extension. While the Bertrand--De Morgan test has been known for a long time, extensions were obtained only recently.

Two different extensions of the Bertrand--De Morgan test are suggested in \cite{D}. One extension appears in Chapter 4, and another one appears in Chapter 7. In Chapter 4, the extension is based on adding Gauss's improvement term, and the main results are given by Theorems 4.1.3 and 4.1.4 there, one for the convergence test and another for divergence. In Chapter 7, another test that is based on an iteration logarithm scheme is provided. The same test as that in \cite[Chapter 7]{D} is provided in \cite{AAT}. The proofs of this last test in both of these studies
\cite{D} and \cite{AAT} are long and technically complicated.  In addition, the proof given in \cite{AAT} has a technical restriction, namely a requirement for existence of the limits (see relation (12) in \cite{AAT} and the further limit relations there); it is also incomplete in the case of divergent series.

In the present note we provide a simple and transparent proof of the extended Bertrand--De Morgan test and demonstrate an application of that test to the theory of birth-and-death processes. The proof is based on reduction to Kummer's test. Kummer's test itself is a universal test that is of special interest in the literature. It covers all positive series \cite{T}, and there are different approaches to this test in the literature (e.g., \cite{D1,Kn,Sa}). Kummer's test is a good subject for verification of new particular tests that may apply to more general classes of series than the known tests.

The note is organized as follows.
In Section \ref{S2}, we recall the Bertrand--De Morgan test and discuss its main restriction. In Section \ref{S3}, we provide a simple proof of the extended Bertrand--De Morgan test.
In Section \ref{S4}, we provide an application of the extended Bertrand--De Morgan test to birth-and-death processes.

\section{The Bertrand--De Morgan test.}\label{S2}
For series \eqref{0}, the Bertrand--De Morgan test is as follows.

\begin{theorem}\label{thm1}
Suppose that the ratio $a_n/a_{n+1}$, where $n$ is large, can be presented in the form
\begin{equation*}\label{1}
\frac{a_n}{a_{n+1}}=1+\frac{1}{n}+\frac{r_n}{n\ln n}.
\end{equation*}
Then series \eqref{0} converges if $\liminf_{n\to\infty}r_n>1$, and it diverges if $\limsup_{n\to\infty}r_n<1$.
\end{theorem}

The  restriction of the test is that the condition for convergence, $\liminf_{n\to\infty}r_n>1$, and divergence, $\limsup_{n\to\infty}r_n<1$, are strongly separated. In particular, if $\lim_{n\to\infty}r_n$ exists, the theorem does not provide information about the situation when $\lim_{n\to\infty}r_n=1$. It follows from the proof of Lemma 4.1 in \cite{A} that if this convergence is from the below, then \eqref{0} diverges. If it is not, then the series may either converge or diverge. The extended Bertrand--De Morgan test that is provided in the next section enables us to test the variety of cases where $\lim_{n\to\infty}r_n=1$.
%
%
%
\section{The extended Bertrand--De Morgan test.}\label{S3}

For formulation of the main result of this note, the following notation is needed. Let $K\geq1$ be an integer, and let $\ln_{(K)}x$ denote the $K$th iterate of natural logarithm, i.e., $\ln_{(1)}x=\ln x$, and for any $2\leq k\leq K$, $\ln_{(k)}x=\ln_{(k-1)}\big(\ln x\big)$.

\begin{theorem}\label{thm2}
Suppose that the ratio $a_n/a_{n+1}$, where $n$ is large, can be presented in the form
\begin{equation}\label{9}
\begin{aligned}
\frac{a_n}{a_{n+1}}=1+\frac{1}{n}+\frac{1}{n}\sum_{i=1}^{K-1}\frac{1}{\prod_{k=1}^{i}\ln_{(k)}n}+\frac{s_n}{n\prod_{k=1}^{K}\ln_{(k)}n}, \quad K\geq1.
\end{aligned}
\end{equation}
Then series \eqref{0} converges if $\liminf_{n\to\infty}s_n>1$, and it diverges if $\limsup_{n\to\infty}s_n<1$. (The empty sum is assumed to be $0$.)
\end{theorem}

The restriction of the test is similar to that given by Theorem \ref{thm1}. If $\lim_{n\to\infty}s_n=1$, then generally we cannot determine convergence or divergence.
Then the additional remainder term written as $s_n^\prime/\big(n\prod_{k=1}^{K+1}\ln_{(k)}n\big)$ determines convergence or divergence of \eqref{0}. According to Theorem \ref{thm2}, if $s_n^\prime\leq1$ for all large $n$, then the series diverges, but if $s_n^\prime>1+\epsilon$, $\epsilon>0$, for all large $n$, it converges. For instance, if the aforementioned remainder term is $o\big[1/\big(n\prod_{k=1}^{K+1}\ln_{(k)}n\big)\big]$
(in particular $O\big[1/\big(n(\ln_{(K+1)}n)^r\prod_{k=1}^{K}\ln_{(k)}n\big)\big]$ with $r>1$),
then \eqref{0} diverges. If, however, it is $O\big[1/\big(n(\ln_{(K+1)}n)^r\prod_{k=1}^{K}\ln_{(k)}n\big)\big]$ with $r<1$, then the series converges (see \cite[Theorem 7.0.10, p. 47]{D}).

Apparently, the set of cases for which one can conclude convergence or divergence of \eqref{0} in Theorem \ref{thm2} is richer than that given by Theorem \ref{thm1}.

\begin{proof} The proof of the theorem is based on Kummer's test, the formulation of which is as follows.

\begin{lemma}\label{lem} Let $\zeta_n$, $n\geq N$, where $N$ is some large number, be a sequence of positive values, and let
$$
\rho_n=\left(\zeta_n\frac{a_n}{a_{n+1}}-\zeta_{n+1}\right),
$$
where positive numbers $a_n$ are defined as in series \eqref{0}. Then series \eqref{0} converges if $\liminf_{n\to\infty}\rho_n>0$, and diverges if $\limsup_{n\to\infty}\rho_n<0$ and $\sum_{n=N}^{\infty}1/\zeta_n=\infty$.
\end{lemma}

We use Kummer's test with
$
\zeta_n=n\prod_{k=1}^{K}\ln_{(k)}n.
$
Let $A$ be an integer for which $\ln_{(K)}A$ is positive. Notice that $\sum_{n=A}^\infty1/\zeta_n$ diverges by comparison with the corresponding integral below. Indeed, let $f_k(x)=\ln_{(k)}x$, $k\geq1$. Then, keeping in mind that
\begin{equation}\label{7}
f_k^\prime(x)=\frac{1}{x\prod_{j=1}^{k-1}\ln_{(j)}x},
\end{equation}
where the empty product is assumed to be $1$,
we obtain
$$
\int_A^\infty\frac{\mathrm{d}x}{x\prod_{k=1}^{K}\ln_{(k)}x}=\int_{A}^{\infty}\mathrm{d}\ln_{(K+1)}x=\ln_{(K+1)}x\Big|_A^\infty=\infty.
$$
Hence,
according to the lemma, series \eqref{0} converges if $\liminf_{n\to\infty}\rho_n>0$, and it diverges if $\limsup_{n\to\infty}\rho_n<0$.

We will need to derive the asymptotic expansions for $\ln_{(k)}(n+1)$ as $n$ tends to infinity. For this purpose we use Taylor's expansion for a twice differentiable function, which in the given case can be written in the form
$f_k(x+\delta)=f_k(x)+\delta f_k^\prime(x)+O(\delta^2f_k^{\prime\prime}(x))$, where $f_k^\prime(x)$ is given by \eqref{7}, and $f_k^{\prime\prime}(x)=O\big(x^{-2}\big)$ as $x\to\infty$.
%
%
%
%
%
For $k\geq1$ we have
\begin{equation}\label{2}
\begin{aligned}
\ln_{(k)}(n+1)
&= \ln_{(k)}(n)+\frac{1}{n\prod_{j=1}^{k-1}\ln_{(j)}n}+O\left(\frac{1}{n^2}\right),
\end{aligned}
\end{equation}
where \eqref{2} is obtained from the above expansion of $f_k(x+\delta)$ for $x=n$ and $\delta=1$.

Now, direct application of expansion \eqref{2} yields the required result.

Indeed, in the case $K=2$, we obtain:
\begin{equation*}
\begin{aligned}
\rho_n=&n\ln n\ln\ln n\frac{a_n}{a_{n+1}}-(n+1)\ln(n+1)\ln\ln(n+1)\\
=& n\ln n\ln\ln n\frac{a_n}{a_{n+1}}-(n+1)\left(\ln n+\frac{1}{n}\right)\left(\ln\ln n+\frac{1}{n\ln n}\right)+o(1)\\
=&n\ln n\ln\ln n\left(\frac{a_n}{a_{n+1}}-1\right)-\ln n\ln\ln n-\ln\ln n-1+o(1)\\
=&n\ln n\ln\ln n\left(\frac{1}{n}+\frac{1}{n\ln n}+\frac{s_n}{n\ln n\ln\ln n}\right)\\
&-\ln n\ln\ln n-\ln\ln n-1+o(1)\\
=&s_n-1+o(1),
\end{aligned}
\end{equation*}
where $o(1)$ includes all smaller terms of the expansion.

In the general case we arrive at the same estimate:
\begin{equation*}\label{5}
\begin{aligned}
\rho_n&= n\prod_{k=1}^{K}\ln_{(k)}n\frac{a_n}{a_{n+1}}-(n+1)\left[\prod_{k=1}^{K}\left(\ln_{(k)}n+\frac{1}{n\prod_{j=1}^{k-1}\ln_{(j)}n}\right)\right]+o(1)\\
&= n\prod_{k=1}^{K}\ln_{(k)}n\left(\frac{a_n}{a_{n+1}}-1\right)-\sum_{j=1}^{K}\prod_{k=1}^{j}\ln_{(K-k+1)}n-1+o(1)\\
&=s_n-1+o(1),
\end{aligned}
\end{equation*}
where all smaller terms are included in the $o(1)$ term.
\end{proof}

\section{Application of Theorem \ref{thm2}.}\label{S4}
In this section, we demonstrate an application of Theorem \ref{thm2} to the theory of birth-and-death processes, improving \cite[Lemma 4.1]{A}. Namely, we prove the following theorem.

\begin{theorem}\label{thm3}
Let the birth and death rates of a birth-and-death process be $\lambda_n$ and $\mu_n$, all in $(0, \infty)$. Then the birth-and-death process is transient if there exist $c>1$, numbers $K\geq1$ and $n_0$ such that for all $n>n_0$
\begin{equation}\label{10}
\frac{\lambda_{n}}{\mu_{n}}\geq1+\frac{1}{n}+\frac{1}{n}\sum_{k=1}^{K-1}\frac{1}{\prod_{j=1}^{k}\ln_{(j)}n}+\frac{c}{n\prod_{j=1}^{K}\ln_{(j)}n},
\end{equation}
and it is recurrent if there exist $K\geq1$ and $n_0$ such that for all $n>n_0$
\begin{equation*}\label{11}
\frac{\lambda_{n}}{\mu_{n}}\leq1+ \frac{1}{n}+\frac{1}{n}\sum_{i=1}^{K}\frac{1}{\prod_{j=1}^{i}\ln_{(j)}n}.
\end{equation*}
\end{theorem}

\begin{proof}
It is known \cite[p. 370]{KM} that a birth-and-death process is recurrent if and only if
\begin{equation}\label{3}
\sum_{n=1}^{\infty}\prod_{k=1}^{n}\frac{\mu_k}{\lambda_k}=\infty.
\end{equation}
Set $a_n=\prod_{k=1}^{n}\mu_k/\lambda_k$. Then \eqref{3} is presented as $\sum_{n=1}^{\infty}a_n=\infty$, and $a_{n}/a_{n+1}=\lambda_{n+1}/\mu_{n+1}$.
Hence, the statement of the theorem follows by application of Theorem \ref{thm2}.
\end{proof}

\begin{acknowledgment}{Acknowledgment.}
The author sincerely thanks the referees for their comprehensive reviews.
\end{acknowledgment}

\begin{affil}
24 Sagan Drive, Cranbourne North, Vic-3977, Australia\\
vabramov126@gmail.com\\
\url{www.vabramov.altervista.org}
\end{affil}

\end{document}